\documentclass[journal,onecolumn]{IEEEtran}

\IEEEoverridecommandlockouts                              

\usepackage{amsmath,amsgen,amstext,amssymb}
\usepackage{amsfonts,amsthm}
\usepackage{amsfonts}
\usepackage{dsfont}
\usepackage{mathrsfs}
\usepackage{latexsym}
\usepackage{times}
\usepackage{graphics}
\usepackage{graphicx}
\usepackage{epsfig}
\usepackage{psfrag}
\usepackage{bm}
\usepackage{lmodern}
\usepackage[T1]{fontenc}
\usepackage{url}
\usepackage{color}

\DeclareMathAlphabet{\mathpzc}{OT1}{pzc}{m}{it}

\newtheorem{theorem}{Theorem}[section]
\newtheorem{proposition}{Proposition}[section]
\newtheorem{lemma}{Lemma}[section]

\newtheorem{remark}{Remark}[section]

\newcommand{\bsigma}{\bm{\sigma}}

\newcommand{\blambda}{\bm{\lambda}}

\newcommand{\bnu}{\bm{\nu}}

\newcommand{\R}{\mathbb R}

\newcommand{\E}{\mathbb{E}}
\newcommand{\ex}{\mathbb{E}}
\newcommand{\pr}{\mathbb{P}}

\newcommand{\var}{\operatorname{Var}}
\newcommand{\norm}[1]{\| #1 \|}

\newcommand{\beqn}{\begin{equation}}
\newcommand{\eeqn}{\end{equation}}

\newcommand{\eqdef}{{:=}}

\providecommand{\norm}[1]{{\lVert#1\rVert}}

\newcommand{\zero}{{\mathbf{0}}}

\newcommand{\bA}{{\mathbf{A}}}

\newcommand{\bP}{{\mathbf{P}}}
\newcommand{\bQ}{{\mathbf{Q}}}

\newcommand{\bU}{{\mathbf{U}}}

\newcommand{\bc}{{\mathbf{c}}}
\newcommand{\be}{{\mathbf{e}}}

\newcommand{\Bm}{{\mathbf{m}}}

\newcommand{\bp}{{\mathbf{p}}}
\newcommand{\br}{{\mathbf{r}}}
\newcommand{\bu}{{\mathbf{u}}}
\newcommand{\bv}{{\mathbf{v}}}

\newcommand{\bx}{{\mathbf{x}}}
\newcommand{\bs}{{\mathbf{s}}}
\newcommand{\bS}{{\mathbf{S}}}
\newcommand{\by}{{\mathbf{y}}}

\newcommand{\g}{\lambda}                

\newcommand{\cF}{{\mathcal{F}}}

\newcommand{\cK}{{\mathcal{K}}}

\newcommand{\cL}{{\mathcal{L}}}

\newcommand{\cM}{{\mathcal{M}}}

\newcommand{\cP}{{\mathcal{P}}}
\newcommand{\cQ}{{\mathcal{Q}}}

\newcommand{\cS}{{\mathcal{S}}}

\newcommand{\IndSet}{\mathcal{I}}
\newcommand{\Int}{\mathbb{Z}}

\usepackage{soul}

\newcommand{\tdel}[1]{{}}
\newcommand{\exq}{\mathbb{E}_{\mathbf{Q}}}

\RequirePackage{amsmath}
\usepackage{amssymb}
\usepackage{algorithm}
\def\ba#1\ea{\begin{align*}#1\end{align*}}
\def\ban#1\ean{\begin{align}#1\end{align}}

\newcommand{\bE}{\ensuremath{\mathds{E}} }

\newcommand{\bvA}{\ensuremath{\overline{{\bf A}}} } 

\newcommand{\vlam}{\ensuremath{{\mbox{\boldmath{$\lambda$}}}} }
\newcommand{\vsig}{\ensuremath{{\mbox{\boldmath{$\sigma$}}}} }
\newcommand{\vnu}{\ensuremath{{\mbox{\boldmath{$\nu$}}}} }
\newcommand{\vzeta}{\ensuremath{{\mbox{\boldmath{$\zeta$}}}} }

\newcommand{\vx}{\ensuremath{{\bf x}} }
\newcommand{\vy}{\ensuremath{{\bf y}} }

\newcommand{\vone}{\ensuremath{{\bf 1}} } 

\newcommand{\lij}{\ensuremath{_{ij}} }
\newcommand{\se}{\ensuremath{^{(\epsilon)}} }
\newcommand{\numin}{\ensuremath{\nu_{\min}} }
\newcommand{\numinp}{\ensuremath{\nu_{\min}'} }

\newcommand{\calk}{{\cal K}}

\newcommand{\calf}{{\cal F}}

\title{\LARGE \bf
On Heavy-Traffic Optimal Scaling of \textit{c}-Weighted MaxWeight Scheduling in Input-Queued Switches
}

\author{Yingdong Lu, Siva Theja Maguluri, Mark S.\ Squillante, Tonghoon Suk
\thanks{Y.~Lu, M.S.~Squillante and T. ~Suk are with Mathematical Sciences, AI Science, IBM Research,
        {\tt\footnotesize \{yingdong, mss, tonghoon.suk\}@us.ibm.com};
	S.T.~Maguluri is with the H.\ Milton Stewart School of Industrial and Systems Engineering, Georgia Institute of Technology,
        {\tt\footnotesize siva.theja@gatech.edu}}%
}

\begin{document}

 \maketitle
\thispagestyle{empty}
\pagestyle{empty}

\begin{abstract}
We consider the asymptotically optimal control of input-queued switches under a cost-weighted variant of MaxWeight scheduling, for which
we establish theoretical properties that include showing the algorithm exhibits optimal heavy-traffic queue-length scaling.
Our results are expected to be of theoretical interest more broadly than input-queued switches.
\end{abstract}


\section{Introduction}
%
Data centers form the backbone of today's big data revolution.
The interchange of data within a data center is facilitated by 
huge $n\times n$ input queued switches (IQSs)
\cite{singh2015jupiter}.
Hence,
understanding scheduling problems in IQSs
is essential 
for real-world
data center networks \cite{perry2014fastpass}.
%
MaxWeight scheduling,
first introduced
for wireless networks
\cite{TasEph92}
and
then
for
IQSs
\cite{MCKEOWN96}, is well-known for being throughput optimal.
However, the issue of delay-optimal scheduling for switches is less clear.
MaxWeight
scheduling
has been shown to be asymptotically optimal in heavy traffic for an objective function of the summation of the squares of queue lengths (QLs)
under
complete resource pooling~\cite{stolyar2004}.
MaxWeight scheduling has also been shown to have optimal scaling in heavy traffic for an objective function of the summation of QLs
under
all
ports
saturated~\cite{maguluri2016},
which was
then extended to the case of incompletely saturated ports~\cite{maguluri2016b}.
Otherwise, the question of delay-optimal scheduling in
IQSs
remains open for general objective functions.

We seek
to gain fundamental insights on the delay-optimal properties of a generalized MaxWeight scheduling policy in $n\times n$
IQSs
in which a linear cost function of QL (delay) is associated with each queue.
Specifically, we extend the results in \cite{maguluri2016} to include per-queue costs and prove that a cost-weighted generalization of MaxWeight scheduling has optimal scaling in heavy traffic
for an objective function consisting of a
linear function of the steady-state average QLs.
Our results shed light on the delay optimality of MaxWeight scheduling and its variants more generally, including extensions to more general objective functions.
In addition, our results are expected to be of theoretical interest beyond
IQS
and related models as implied by our extension of the drift method,
first introduced in \cite{ErySri12} and together with its subsequent
developments.
This paper extends an earlier version~\cite{LuMaSq+18} to include a tighter universal lower bound (l.b.) on the average weighted QL and an explicit expression
for the weighted sum of QLs in heavy traffic in general $n\times n$
IQSs.

%
\S\ref{sec:model} presents our mathematical model and formulation, and
\S\ref{sec:MaxWeight} presents our analysis of cost-weighted MaxWeight scheduling,
followed by conclusions and some proofs.

\section{Model and Formulation}\label{sec:model}
Consider an
IQS
with $n$ input ports and $n$ output ports.
Each input port has a queue associated with every output port that stores packets waiting to be transmitted to the output port.
Let $(i,j) \in \IndSet \eqdef \{ (i,j) : i,j \in [n] \}$, $[n] \eqdef \{ 1, \ldots, n \}$, index the queue associated with the $i$th input port and the $j$th output port.
Let $c_{ij}>0$ denote the cost associated with queue $(i,j)$ and define $\bc := (c_{ij}) \in \R^{n^2}$.
Further define a new inner product on ${\mathbb R}^{n^2}$ with respect to (w.r.t.) the vector $\bc$ as follows
\begin{equation}
\label{eq:dot_product}
\langle x, y \rangle_\bc \; \eqdef \; \sum_{ij} c_{ij}x_{ij}y_{ij} .
\end{equation}
Hence, the corresponding norm of a vector $\bx \in \R^{n^2}$ is given by $||\bx||_{\bc}^2=\sum_{ij}c\lij x^2\lij $.

Packets arrive at queue $(i,j)$ from
a stochastic process.
Time is slotted
and denoted by
$t\in\Int_+\eqdef\{0,1,\ldots\}$.
At each time $t$, a scheduling policy selects a set of queues from which to simultaneously transmit packets under the constraints:
(1) At most one packet can be transmitted from an input port;
(2) At most one packet can be transmitted to an output port.
We refer to a \emph{schedule} as a subset of queues that satisfies these constraints.

A schedule is formally described by an $n^2$-dimensional binary vector $\bs=(s_{ij})_{(i,j)\in\IndSet}$ such that $s_{ij}=1$ if queue $(i,j)$ is in the schedule,
and $s_{ij}=0$ otherwise.
Let $\cP$ denote the set of all maximal schedules, i.e.,
\begin{equation*}
\cP  = \left\{
\big[ \bs\in\{0,1\}^{n^2} \big] \; : \; \begin{array}{l} \sum_{j\in[n]} s_{ij}= 1, \; \forall i\in [n]\\ \sum_{i\in[n]} s_{ij}= 1, \; \forall j\in [n] \end{array}
\right\} ,
\end{equation*}
and $\bS(t)\in \cP$ the schedule for period $t$ under the $\bc$-weighted MaxWeight scheduling algorithm defined below.
Let $Q_{ij}(t) \in \Int_+$ denote the length of the infinite capacity queue $(i,j)$ at time $t$ under this MaxWeight policy
and $A_{ij}(t) \in \Int_+$ the number of arrivals to queue $(i,j)$ during $[t,t+1)$.
The queueing dynamics then can be expressed as
\begin{equation}\label{eq:dynamics_of_Q-wU}
Q_{ij}(t+1) \; = \; Q_{ij}(t) + A_{ij}(t) - S_{ij}(t) + U_{ij}(t) ,
\end{equation}
where $U_{ij}(t) \in \Int_+$ denotes the unused service for queue $(i,j)$ at time $t$.
Any selected schedule is always a maximal schedule in $\cP$,
which results in an unused service at those queues with no packets to serve.
We assume that $\{ A_{ij}(t) : t \in \Int_+, \, (i,j) \in \IndSet \}$ are independent random variables (r.v.s)
and that, for each fixed $(i,j) \in \IndSet$, $\{ A_{ij}(t) : t \in \Int_+ \}$ are identically distributed with $\E[A_{ij}(t)] = \lambda_{ij}$.
Define $\bQ(t) \eqdef (Q_{ij}(t))_{(i,j)\in\IndSet}$, $\bA(t) \eqdef (A_{ij}(t))_{(i,j)\in\IndSet}$,
$\bS(t) \eqdef (S_{ij}(t))_{(i,j)\in\IndSet}$ and $\bU(t) \eqdef (U_{ij}(t))_{(i,j)\in\IndSet}$.

Consider the above
IQS
model under the $\bc$-weighted MaxWeight scheduling Algorithm~\ref{alg:maxwt}.
\begin{algorithm}
	\caption{\label{alg:maxwt} $\bc$-Weighted MaxWeight Scheduling}
	Let $\bc\in \R^{n^2}$ be a given positive weight (cost) vector, i.e., $c\lij > 0, \; \forall i,j$.
	Then, in every time slot $t$ under the $\bc$-weighted MaxWeight algorithm, each queue is assigned a weight $c_{ij}Q_{ij}(t)$ and a schedule with the maximum weight is chosen,
	namely
	\begin{align*}
	\bS(t) = \arg\max_{\bs \in \cP} \sum\lij c_{ij}Q_{ij}(t) s\lij =
	\arg \max_{\bs \in \cP}
	\langle \bQ(t),\bs  \rangle_\bc .
	\end{align*}
	Ties are broken uniformly at random.
\end{algorithm}

The objective function consists of minimizing a weighted summation of expected delay cost in steady state,
based on which we
establish
delay-optimal properties of the $\bc$-weighted MaxWeight scheduling algorithm.
Given the relationship between delays and QLs via Little's Law, we henceforth focus on cost as a function of the QLs.
Suppose the QL process $\bQ^\pi(t)$ under any
stationary policy $\pi$ converges in distribution to a steady state random vector $\overline{\bQ}^\pi$.
The objective function
can then be expressed as
\begin{equation*}
\min_{\pi \in \cM} \quad \E\bigg[ \sum_{(i,j)\in\IndSet} c_{ij}\overline{Q}^\pi_{ij} \bigg] ,
\end{equation*}
where $\cM$ denotes the set of all stationary policies.
Note that \cite{maguluri2016} considers the specific case of $c_{ij}=1, \forall i,j$.

\section{Heavy Traffic Analysis}
\label{sec:MaxWeight}
We study the switch system when the arrival rate vector $\vlam$ approaches a point on the boundary of the capacity region such that all the ports are saturated.
In other words, we consider the arrival rate vector approaching the face $\calf$ of the capacity region where
\begin{align*}
\calf
     &=\bigg\{\vlam\in {\mathbb R}_+^{n^2} : \langle \vlam, \be^{(i)}_{\bc} \rangle_{\bc} =1, \langle \vlam, {\tilde \be}^{(j)}_{\bc} \rangle_{\bc} =1, \; \forall i, j \in [n] \bigg\} ,
\end{align*}
and where $\be_{\bc}^{(i)}=\{\bx\in {\mathbb R}^{n^2}, x_{ij} = \frac{1}{c_{ij}}, x_{i'j}=0, \forall i'\neq i\}$ and
${\tilde \be}_{\bc}^{(j)}=\{\bx\in {\mathbb R}^{n^2}, x_{ij} = \frac{1}{c_{ij}}, x_{ij'}=0, \forall j'\neq j\}$.

We will obtain an exact expression for the heavy traffic scaled weighted sum of QLs under the $\bc$-weighted MaxWeight algorithm in heavy traffic,
along similar lines as \cite{maguluri2016b} but with the dot product redefined in \eqref{eq:dot_product} and related technical differences.
To obtain the desired result for heavy traffic performance under the $\bc$-weighted MaxWeight algorithm,
we first provide a universal l.b.\ on the average weighted QL.
We then establish that the QL vector concentrates close to a lower dimensional cone in the heavy traffic limit, which is called state space collapse.
Finally, we exploit this state space collapse result to obtain an exact expression for the heavy traffic scaled weighted sum of QLs in heavy traffic.
The proofs of the main results in \ref{sec:SSC} and \ref{sec:heavy} follow a similar logical approach to that in \cite{maguluri2016},
though with important technical differences and details for the \emph{$\bc$-weighted} MaxWeight algorithm due to the
modified dot product, norms and projections w.r.t.\ $\bc$.

Throughout, we consider a \emph{base family of switch systems} having arrival processes $\bA^{(\epsilon)}(t)$ parameterized by $0<\epsilon<1$
such that the mean arrival rate vector is given by $\blambda^{(\epsilon)}=\E[\bA^{(\epsilon)}(t)]=(1-\epsilon)\bnu$ for some
$\bnu$ in the relative interior of $\cF$
with $\numin\eqdef \min\lij \nu\lij >0$, and the arrival variance vector is given by $\var(\bA^{(\epsilon)})=(\bsigma^{(\epsilon)})^2<\infty$.

\subsection{Universal Lower Bound}
Consider a priority queueing system $\tilde{\bP}^{(\ell)}$ under a fixed priority ordering $\bp^{(\ell)} \in \cL$ among all $L=n!$ schedules
in the set $\cP$, indexed by $\ell$, where $\cL$ is the set of all possible priority orderings of the $L$ schedules.
Let $\tilde{Q}_{\ell,l}(t)$, $l=1, \ldots, L$, denote the QL process of the $l$th highest priority class in the system $\tilde{\bP}^{(\ell)}$
under ordering $\bp^{(\ell)}$.
Let $\Bm^{(\ell)}(l)$ be the set of queues $(i,j)$ of the switch contained within the $l$th priority class in $\bp^{(\ell)}$,
and $\tilde{A}_{\ell,l}(t)$ the composite arrival r.v.\ from all $A_{ij \in \Bm^{(\ell)}(l)}(t)$ that leads to the smallest QL among the queues
$(i,j)$ in the set $\Bm^{(\ell)}(l)$.
Then, for the system $\tilde{\bP}^{(\ell)}$, we can write an expression for the QL process of the highest priority class $1$ as
\begin{align*}
[\tilde{Q}_{\ell,1}(t+1)]^2-[\tilde{Q}_{\ell,1}(t)]^2 
= & [\tilde{Q}_{\ell,1}(t) + \tilde{A}_{\ell,1}(t+1) - 1 +\tilde{V}_{\ell,1}(t+1)]^2 -[\tilde{Q}_{\ell,1}(t)]^2 \\
=& [\tilde{Q}_{\ell,1}(t) + \tilde{A}_{\ell,1}(t+1) - 1]^2 -\tilde{V}_{\ell,1}^2(t+1) - \tilde{Q}_{\ell,1}^2(t)\\
=&  [\tilde{A}_{\ell,1}(t+1) - 1]^2 + 2 \tilde{Q}_{\ell,1}(t) [\tilde{A}_{\ell,1}(t+1) - 1] -\tilde{V}_{\ell,1}^2(t+1)
\end{align*}
where $\tilde{V}_{\ell,u}(t)$ denotes the time spent serving all lower priority classes $v>u$ and idling.
From the relationship
\begin{align*}
[\tilde{Q}_{\ell,1}(t) + \tilde{A}_{\ell,1}(t+1) - 1 +\tilde{V}_{\ell,1}(t+1)]\tilde{V}_{\ell,1}(t+1)=0 ,
\end{align*}
we therefore have
\begin{align*}
[\tilde{Q}_{\ell,1}(t) + \tilde{A}_{\ell,1}(t+1) - 1]\tilde{V}_{\ell,1}(t+1)=- \tilde{V}_{\ell,1}^2(t+1).
\end{align*}
Similarly, for the next highest priority class $2$, we obtain
\begin{align*}
[\tilde{Q}_{\ell,2}(t+1)]^2-[\tilde{Q}_{\ell,2}(t)]^2 &=  [\tilde{A}_{\ell,2}(t+1) - \tilde{V}_{\ell,1}(t+1)]^2 
+ 2 \tilde{Q}_{\ell,2}(t) [\tilde{A}_{\ell,1}(t+1) - \tilde{V}_{\ell,1}(t+1)] -\tilde{V}_{\ell,2}^2(t+1) ,
\end{align*}
thus rendering in stationarity
\begin{equation*}
\ex[(1-\tilde{A}_{\ell,1})^2] - 2 \ex[\tilde{Q}_{\ell,1}(1-\tilde{A}_{\ell,1})] -\ex[\tilde{V}_{\ell,1}^2]=0 .
\end{equation*}
Hence, $\tilde{Q}_{\ell,1}$ will be finite, and more specifically
\begin{align*}
\ex[\tilde{Q}_{\ell,1}] = \frac{\ex[(1-\tilde{A}_{\ell,1})^2]  -\ex[\tilde{V}_{\ell,1}^2]}{1-\ex[\tilde{A}_{\ell,1}]} \le \frac{\ex[(1-\tilde{A}_{\ell,1})^2] }{1-\ex[\tilde{A}_{\ell,1}]},
\end{align*}
which then yields for $\tilde{Q}_{\ell,2}$ in stationarity
\begin{align*}
2\ex[\tilde{Q}_{\ell,2}] \ex[ \tilde{V}_{\ell,1}-\tilde{A}_{\ell,2} ] \ge \ex[\tilde{A}_{\ell,2}^2] - 2\ex[\tilde{A}_{\ell,2}]\ex[\tilde{V}_{\ell,1}]- \ex[\tilde{V}_{\ell,2}^2].
\end{align*}
Continuing in this manner, we have in general for class $l$
\begin{align*}
\ex[\tilde{A}_{\ell,l} -\tilde{V}_{\ell,l-1}^2] - 2\epsilon \ex[\tilde{Q}_{\ell,l}] - \ex[\tilde{V}_{\ell,l}^2]=0 , \quad \forall l=2,\ldots,L .
\end{align*}
Upon expanding the first term, we obtain
\begin{equation}
\ex[\tilde{Q}_{\ell,l}] \ge \frac{\ex[\tilde{A}_{\ell,l}^2] - 2\ex[\tilde{A}_{\ell,l}]\ex[\tilde{V}_{\ell,l-1}]- \ex[\tilde{V}_{\ell,l}^2]}{2\epsilon} ,
\label{eq:l-queue-epsilon}
\end{equation}
which, since we know $\ex[\tilde{V}_{\ell,l}^2]=O(\epsilon)$, renders
\begin{equation}
\liminf_{\epsilon\downarrow 0} \epsilon \ex[\tilde{Q}_{\ell,l}] \ge \ex[\tilde{A}_{\ell,l}^2] - 2\ex[\tilde{A}_{\ell,l}]\ex[\tilde{V}_{\ell,l-1}].
\label{eq:l-queue}
\end{equation}
Let $\tilde{\cQ}_{\ell,l}^{(\epsilon)}$ and $\tilde{\cQ}_{\ell,l}$ be the RHS of \eqref{eq:l-queue-epsilon} and
\eqref{eq:l-queue}, respectively.
Define
$\hat{\cQ}_{\ell,ij}^{(\epsilon)} := \min_{l : ij \in \Bm^{(\ell)}(l)} \tilde{\cQ}_{\ell,l}^{(\epsilon)}$ and
$\hat{\cQ}_{\ell,ij}              := \min_{l : ij \in \Bm^{(\ell)}(l)} \tilde{\cQ}_{\ell,l}$.
We then have the desired universal l.b.
%
\begin{proposition} \label{prop:ULB}
Consider the base family of switches and fix a scheduling policy under which the system is stable for any $0<\epsilon<1$.
Suppose the QL process $\bQ^{(\epsilon)}(t)$ converges in distribution to a steady state random vector $\overline{\bQ}^{(\epsilon)}$,
and assume $(\bsigma^{(\epsilon)})^2\to\bsigma^2$.
Define
\begin{equation*}
\hat{\cQ}_{*}^{(\epsilon)} := \min_{\bp^{(\ell)} \in \cL} \sum_{ij} c_{ij} \hat{\cQ}_{\ell,ij}^{(\epsilon)} , \quad
\hat{\cQ}_{*}              := \min_{\bp^{(\ell)} \in \cL} \sum_{ij} c_{ij} \hat{\cQ}_{\ell,ij} .
\end{equation*}
Then, for each of these switch systems, the average weighted QL is lower bounded by
$\E[ \sum_{i,j}c_{ij}\overline{Q}_{ij}^{(\epsilon)} ] \geq \hat{\cQ}_{*}^{(\epsilon)}$,
and, in the heavy-traffic limit as $\epsilon\downarrow 0$, we have
\begin{equation}
\liminf_{\epsilon\downarrow 0} \epsilon \E\bigg[ \sum_{i,j}c_{ij}\overline{Q}_{ij}^{(\epsilon)} \bigg] \;\; \geq \;\; \hat{\cQ}_{*} .
\label{eq:prop:HT}
\end{equation}
\end{proposition}
\begin{proof}
The overall average QL $\sum_{l} \ex[\tilde{Q}_{\ell,l}]$ for each $L$-class priority queueing system $\tilde{\bP}^{(\ell)}$ under ordering $\bp^{(\ell)}$,
$\forall \bp^{(\ell)} \in \cL$, forms the vertices of the performance region polytope in which must lie the overall average QL of any scheduling policy in
the $L$-class queueing system.
Since, by construction, the $l$th queue under any $\bp^{(\ell)} \in \cL$ can be scheduled whenever at least one queue $(i,j)$ in $\Bm^{(\ell)}(l)$ has a packet,
this polytope together with $\hat{\cQ}_{\ell,ij}^{(\epsilon)}$ and $\hat{\cQ}_{\ell,ij}$ provide a l.b.\ on the overall average QL of any scheduling policy
in the original switch system.
It follows that the average weighted QL under any scheduling policy in the switch is lower bounded by $\hat{\cQ}_{*}^{(\epsilon)}$,
with the corresponding heavy-traffic limit lower bounded by $\hat{\cQ}_{*}$.
\end{proof}
\begin{remark}
The above l.b.\ \eqref{eq:prop:HT} improves upon the looser bound of $c_{\min}(\norm{\bsigma}^2/2)$ established in~\cite{LuMaSq+18}.
\end{remark}

\subsection{State Space Collapse}\label{sec:SSC}
Since \cite{maguluri2016} considers $c_{ij}=1, \forall i,j$, the state space collapse in our general case is to a different cone.
To establish this state space collapse result,
we first define the cone $\cK_{\bc}$ to be the cone spanned by the vectors $\be_{\bc}^{(i)}$ and ${\tilde \be_{\bc}}^{(j)}$, namely
\begin{align*}
\calk_{\bc} \eqdef \bigg\{ \vx\in {\mathbb R}^{n^2}: x_{ij} = \frac{w_i+{\tilde w}_j}{c_{ij}}, \quad w_i, {\tilde w}_j \in {\mathbb R}_+\bigg\}.
\end{align*}
For any $\vx\in {\mathbb R}^{n^2}$, define
$\vx_{\parallel \cK_{\bc}} \eqdef \arg\min_{\vy\in \calk_{\bc} } ||\vx-\vy||_{\bc}$
to be the projection of $\vx$ onto the cone $\cK_{\bc}$.
The error after projection is denoted by
$\vx_{\bot\cK_{\bc}}=\vx-\vx_{\parallel \cK_{\bc}}$.
To simplify the notation throughout the paper, we will write $\vx_{\parallel_{\bc}}$ to mean $\vx_{\parallel \cK_{\bc}}$ and
write $\vx_{\bot_{\bc}}$ to mean $\vx_{\bot\cK_{\bc}}$.
Let $\cS_{\bc}$ denote the space spanned by the cone $\cK_{\bc}$, or more formally
\begin{align*}
\cS_{\bc} = \bigg\{ \vx\in {\mathbb R}^{n^2}: x_{ij} = \frac{w_i+{\tilde w}_j}{c_{ij}}, \quad w_i, {\tilde w}_j \in {\mathbb R}\bigg\}.
\end{align*}
The projection of $\vx\in {\mathbb R}^{n^2}$ onto the space $\cS_{\bc}$ is denoted by $\vx_{\parallel \cS_{\bc}}$,
with the error after projection denoted by $\vx_{\bot\cS_{\bc}}$.

Now, consider the base family of switch systems under the $\bc$-weighted MaxWeight scheduling algorithm with the maximum possible arrivals in any queue denoted by $A_{\max}$.
Let the variance of the arrival process be such that $\|\vsig^{(\epsilon)}\|^2 \leq \widetilde{\sigma}^2$ for some $\widetilde{\sigma}^2$ that is not dependent on $\epsilon$.
Let $\overline{\bQ}^{(\epsilon)}$ denote the steady state random vector of the QL process for each switch system parameterized by $\epsilon$.
We then have the following proposition.
%
\begin{proposition} \label{prop:SSC}
	For each system
above with $0< \epsilon \leq \numinp$, the steady state QL vector satisfies
	\ba
	\ex \bigg [\| \overline{\bQ}_{\bot_{\bc}} \se \|^r\bigg ] \leq (M_r)^r , \qquad \forall r \in\{1,2,\ldots\},
	\ea
	where $\numinp$  and $M_r $ are functions of $r,\widetilde{\sigma},\vnu,A_{\max},\numin$ but independent of $\epsilon$.
\end{proposition}
\begin{proof}
Omitting superscript $(\epsilon)$ to simplify the notation and clarify the presentation, 
our general approach consists of defining a Lyapunov function $W_{\bot_{\bc}}(\overline{\bQ}) \eqdef \|\overline{\bQ}_{\bot_{\bc}}\|_{\bc}$
and its drift $\Delta W_{\bot_\bc}(\overline{\bQ}) \eqdef \big(W_{\bot_{\bc}}(\bQ(t+1))-W_{\bot_{\bc}}(\bQ(t))\big)\mathbb I_{\{\bQ(t)=\overline{\bQ}\}}$,
for all $\overline{\bQ} \in {\mathbb R}^{n^2}$.
Then, from Lemma~\ref{lem:Lyapunov-drift} in Appendix~\ref{app:SSC},
there exist positive numbers $\eta$, $\kappa$ and $D$ that depend on $\widetilde{\sigma}$, $\vnu$, $A_{\max}$ and $\numin$, but not on $\epsilon$ such that
\begin{align*}
	\ex[\Delta W_{\bot_{\bc}}(\overline{\bQ}) |\bQ(t) =\overline{\bQ} ] &\le -\eta, \qquad \forall \overline{\bQ} ,\, W_{{\bot}_{\bc}}(\overline{\bQ})\geq \kappa, \\
	\pr[|\Delta W_{\bot_{\bc}}(\overline{\bQ})|\leq D] &= 1, \qquad\quad \forall \overline{\bQ},
\end{align*}
from which we derive, by Lemma 3 in \cite{maguluri2016},
\begin{align*}
	\ex \bigg[\| \overline{\bQ}_{\bot_{\bc}} \se \|^r \bigg] &\leq (2\kappa)^4 + r\!\bigg(\frac{D+\eta}{\eta}\bigg)^r(4D)^r \\
	&\leq (2\kappa)^r + \sqrt{r}e\bigg(4D\frac{r}{e}\frac{D+\eta}{\eta}\bigg)^r\\
	&\leq 2\bigg(\max\bigg\{ 2\kappa,\;(\sqrt{r}e)^{1/r}4D\frac{r}{e}\frac{D+\eta}{\eta} \bigg\}\bigg)^r \\
	&= (M_r)^r={\bigg(2^{1/r}\max\Big\{ 2\kappa,\;(\sqrt{r}e)^{1/r}4D\frac{r}{e}\frac{D+\eta}{\eta} \Big\}\bigg)}^r,
\end{align*}
which is a function of $r$, $\widetilde{\sigma}$, $\vnu$, $A_{\max}$ and $\numin$,
but independent of $\epsilon$,
hence completing the proof.
\end{proof}

\begin{remark}
The special case of $\bc=\vone$ renders the standard MaxWeight algorithm and our results coincide with the state space collapse in \cite{maguluri2016}.
More generally, the capacity region and maximal face $\cF$ are not dependent on the choice of the weight vector $\bc$.
However, for any positive weight vector, the state space collapses into the normal cone of the face $\cF$ w.r.t.\ the dot product defined by the weight vector $\bc$.
This cone depends upon the choice of $\bc$, and thus the choice of the weight vector ``tilts'' the cone of collapse.
\end{remark}

\subsection{Weighted Sum of Queue Lengths in Heavy Traffic}\label{sec:heavy}
We next exploit the above state space collapse result to obtain an exact expression for the heavy traffic scaled weighted sum of QLs in heavy traffic.
Our main results are provided in the following theorem, with the next section providing a general matrix solution approach to calculate the corresponding
limit and obtain an explicit expression for this heavy traffic limit.

\begin{theorem}\label{thm:Heavytraffic}
Consider the base family of
switches
under
the
$\bc$-weighted MaxWeight
algorithm as in Proposition \ref{prop:SSC}.
	Then, in the heavy traffic limit as $\epsilon\downarrow 0$, we have
	\begin{equation}
	\lim_{\epsilon\to 0} \epsilon \ex\bigg[  \sum_{ij} c\lij \overline{Q}\lij^{(\epsilon)} \bigg] = \frac{n}{2}\bigg \langle \vsig^2, \vzeta  \bigg \rangle_{\bc},
	\label{eq:thm:HTL}
	\end{equation}
where $\vsig^2=\bigg(\sigma^2_{ij}\bigg)_{ij}$,
and the vector \vzeta is defined by
$\zeta\lij \eqdef \norm{(\be_{ij})_{||\cS_{\bc}}}_{\bc}^2$
and the matrix $\be_{ij}$ by $1$ in position $(i,j)$ and $0$ elsewhere.
\end{theorem}
\begin{proof}
We again omit the superscript $(\epsilon)$ to simplify the notation and clarify the presentation.
Let $\bA$ denote the arrival vector in steady state, which is distributed identical to the random vector $\bA(t)$ for any
$t$.
Further let $\bS(\overline{\bQ})$ and $\bU(\overline{\bQ})$ denote the steady state schedule and unused service vector, respectively,
both of which depend on the QL vector in steady state $\overline{\bQ}$.
Recalling the queueing dynamics in \eqref{eq:dynamics_of_Q-wU},
define $\overline{\bQ}^+\eqdef \overline{\bQ}+\bA-\bS(\overline{\bQ})+\bU(\overline{\bQ})$ to be the QL vector at time $(t+1)$,
given the QL vector at time $t$ is $\overline{\bQ}$.
Clearly, $\overline{\bQ}^+$ and $\overline{\bQ}$ have the same distribution.

The proof proceeds by setting the drift of the Lyapunov function $V(\bQ)=\norm{\bQ_{||\cS_{\bc}}}_{\bc}^2$ to $0$ in steady state, from which we obtain
\begin{align*}
	0 =& \ex[V(\overline{\bQ}^+)-V(\overline{\bQ}) ] 
	=& \ex[\norm{(\bA-\bS(\bQ))_{||\cS_{\bc}} }_{\bc}^2+ 2\langle \bQ_{||\cS_{\bc}}, (\bA-\bS(\bQ))_{||\cS_{\bc}} \rangle_{\bc} 
	- \norm{\bU_{||\cS_{\bc}}(\bQ)}_{\bc}^2 + 2\langle \bQ^+_{||\cS_{\bc}}, \bU_{||\cS_{\bc}}(\bQ)\rangle_{\bc}].
\end{align*}
This yields an equation of the form
\begin{align*}		
& 2\ex\Big[\big \langle \bQ_{||\cS_{\bc}}, (\bS(\bQ)-\bA)_{||\cS_{\bc}}\big \rangle_{\bc}\Big]
 = \ex\Big[\norm{(\bA-\bS(\bQ))_{||\cS_{\bc}} }_{\bc}^2\Big] \nonumber 
-\ex\Big[\norm{\bU_{||\cS_{\bc}}(\bQ)}_{\bc}^2\Big]
+2\ex\Big[\big \langle\bQ_{||\cS_{\bc}}^+, \bU_{||\cS_{\bc}}(\bQ)\big \rangle_{\bc}\Big].
\end{align*}
The desired result then follows from Lemmas~\ref{lem:LHS} and \ref{lem:RHS} in Appendix~\ref{app:3.1},
matching the LHS and RHS of \eqref{eq:thm:HTL}.
\end{proof}

\subsection{Explicit Expression for Heavy Traffic Limit}
Given the important differences in the cone for our general case of $c_{ij} > 0$ in comparison to the cone in \cite{maguluri2016},
we now seek to obtain an explicit expression for the RHS of \eqref{eq:thm:HTL}.
More specifically, we want to calculate $\zeta_{ij}$ for each $(i,j)$.
To start, let us consider the following $n^2$-dimensional vectors, given in matrix form for any $i, j\in [n-1]$,
\begin{align*}
B_{ij} =
\begin{pmatrix} E_{ij} & -E_i \\ -E^T_j & 1 \end{pmatrix}
\end{align*}
where
$E_{ij}$ is an $(n-1)\times (n-1)$ matrix with the $(i,j)$th element $1$ and all other elements $0$,
$E_i$ is an $(n-1)$-dimensional column vector with the $i$th element $1$ and all other elements $0$,
and superscript $T$ denotes the transpose operator.
These $B_{ij}$ are certainly linearly independent.
At the same time, it can be readily verified that the $\bc$-inner product of $B_{ij}$ with any of $\be_c^{(1)}, \ldots, \be_c^{(n)}, {\tilde \be}_c^{(1)}, \ldots, {\tilde \be}_c^{(n-1)}$ is $0$.
In other words, while $\be_c^{(1)}, \ldots, \be_c^{(n)}, {\tilde \be}_c^{(1)}, \ldots, {\tilde \be}_c^{(n-1)}$ forms an affine basis for the $\parallel_\bc$-space,
$B_{ij}$ forms an affine basis of the $\bot_\bc$-space.  
Recall that $\zeta\lij =\norm{(\be_{ij})_{||\cS_{\bc}}}_{\bc}^2$.
Our approach consists of calculating the projection of $\be_{ij}$ to the $\bot_\bc$-space, which in turn will render its projection to the $\parallel_\bc$-space. 
More specifically, for each $\be_{ij}$, we will find a vector that is in the $\bot_\bc$-space and has the exact same $\bc$-inner product with every $B_{ij}$.
We now carry out these calculations. 

It is clear that $\langle \be_{ij}, B_{k\ell} \rangle_\bc =c_{ij}$ only when $i=k$ and $j=\ell$, and is $0$ otherwise.
Without loss of generality, let us consider $\be_{11}$.
The basic idea is to find the vector $(\be_{11})_{\bot\cS_{\bc}}$ that is perpendicular to all $B_{ij}$ except $B_{11}$, as well as to $\be_c^{(i)}$ and ${\tilde \be}_c^{(j)}$.
From the special structure of $B_{ij}$,
and for some $x_1, x_2,\ldots, x_{n-1}, y_1, y_2, \ldots, y_{n-1}, z$,
it can be readily verified that
\begin{align}
\label{eqn:bot_projection}
(\be_{11})_{\bot\cS_{\bc}}= \left( \begin{array}{ccccc}
1+\frac{-z -x_1-y_1 }{c_{11}}& \frac{-z-x_1-y_2}{c_{12}} & \ldots & \frac{-z-x_1-y_{n-1}}{c_{1,n-1} }& -\frac{x_1}{c_{1n}}\\
\frac{-z -x_2-y_1}{c_{21}}& \frac{-z-x_2-y_2}{c_{22} }& \ldots & \frac{-z-x_2-y_{n-1}}{c_{2, n-1}} & -\frac{x_2}{c_{2n}} \\
\frac{-z -x_3-y_1}{c_{31}} & \frac{-z-x_3-y_2}{c_{32}} & \ldots & \frac{-z-x_3-y_{n-1}}{c_{3,n-1}} & -\frac{x_3}{c_{3n}} \\
& & \ldots & & \\
\frac{-z -x_{n-1}-y_1}{c_{n-1,1}} & \frac{-z-x_{n-1}-y_2}{c_{n-1,2}} & \ldots & \frac{-z-x_{n-1}-y_{n-1}}{c_{n-1,n-1}} & -\frac{x_{n-1}}{c_{n-1,n}} \\
-\frac{y_1}{c_{n1} }&- \frac{y_2}{c_{n2}} & \ldots & -\frac{y_{n-1}}{c_{n,n-1} }& \frac{z}{c_{nn}}
\end{array}\right)
\end{align}
to ensure that $\langle (\be_{11})_{\bot\cS_{\bc}}, B_{11}\rangle_\bc=c_{11}$, and $\langle (\be_{11})_{\bot\cS_{\bc}}, B_{ij}\rangle_\bc=0$ for all other $(i, j)$.
Furthermore, $\langle (\be_{11})_{\bot\cS_{\bc}}, \be_c^{(i)}\rangle_\bc=0$ and $\langle (\be_{11})_{\bot\cS_{\bc}}, {\tilde \be}_c^{(j)}\rangle_\bc=0$ yield
the following $2n-1$ linear equations in terms of $x_1, x_2,\ldots, x_{n-1}, y_1, y_2, \ldots, y_{n-1}, z$:
\begin{align*}
1+\frac{ -z-x_1-y_1 }{c_{11}} + \frac{ -z-x_1-y_2 }{c_{12}} + \ldots + \frac{ -z-x_1-y_{n-1} }{c_{1,n-1}} - \frac{x_1}{c_{1n}}&=0,\\
\frac{ -z-x_2-y_1 }{c_{21}} + \frac{ -z-x_2-y_2 }{c_{22}} + \ldots + \frac{ -z-x_2-y_{n-1} }{c_{2,n-1}} - \frac{x_2}{c_{2n}}&=0,\\
& \hspace*{-1.75in} \cdots \\
\frac{ -z-x_{n-1}-y_1 }{c_{n-1,1}} + \frac{ -z-x_{n-1}-y_2 }{c_{n-1,2}} + \ldots+\frac{ -z-x_{n-1}-y_{n-1} }{c_{n-1,n-1}} -\frac{x_{n-1}}{c_{n-1,n}} &= 0,\\
\frac{y_1}{c_{n1}} + \frac{y_2}{c_{n2}} + \ldots + \frac{y_{n-1}}{c_{n,n-1}} -\frac{z}{c_{nn}} &=0, \\
1+\frac{ -z-x_1-y_1 }{c_{11}} + \frac{ -z-x_2-y_1 }{c_{21}} + \dots+ \frac{ -z-x_{n-1}-y_1 }{c_{n-1,1}} - \frac{y_1}{c_{n1} }&=0, \\
\frac{ -z-x_1-y_2 }{c_{12}} + \frac{ -z-x_2-y_2 }{c_{22}} + \dots+ \frac{ -z-x_{n-1}-y_2 }{c_{n-1,2}} - \frac{y_2}{c_{n2} }&=0, \\
& \hspace*{-1.75in} \cdots \\
\frac{ -z-x_1-y_{n-1} }{c_{1,n-1}} + \frac{ -z-x_2-y_{n-1} }{c_{2, n-1}} + \dots+ \frac{ -z-x_{n-1}-y_{n-1} }{c_{n-1,n-1}} - \frac{y_{n-1}}{c_{n, n-1} }&=0.
\end{align*}

This can be written in matrix form as $G (x_1, x_2,\ldots, x_{n-1}, y_1, y_2, \ldots, y_{n-1}, z)^T= \be^{(2n-1)}_{11}$,
where $G$ is a $(2n-1)\times (2n-1)$ nonsingular matrix representing the coefficients from the above system of equations,
and $\be^{(2n-1)}_{11}$ is a $(2n-1)$-dimensional vector with the first and $(n+1)$st elements being $1$, and all others being $0$.
More specifically, we have the following block matrix expression for the matrix $G$
\begin{align*}
 G= \left( \begin{array}{ccc} D_1 & C &H_1 \\ \zero & H_3^T & -c_{nn}^{-1} \\ C^T & D_2 & H_2 \end{array} \right),
 \end{align*}
with the $(n-1) \times (n-1)$ matrices
$C=\left[\frac{1}{c_{ij}}\right]_{i,j=1,\ldots,n-1}$,
$D_1=\mbox{diag}\left[\sum_{j=1}^{n}\frac{1}{c_{1j}},\sum_{j=1}^{n}\frac{1}{c_{2j}},\ldots,\sum_{j=1}^{n}\frac{1}{c_{n-1,j}}\right]$ and
$D_2=\mbox{diag}\left[\sum_{i=1}^{n}\frac{1}{c_{i1}},\sum_{i=1}^{n}\frac{1}{c_{i2}},\ldots,\sum_{i=1}^{n}\frac{1}{c_{i,n-1}}\right]$,
the $(n-1)$-dimensional column vectors
$H_1 = \left[\sum_{j=1}^{n-1}\frac{1}{c_{1j}}, \ldots, \sum_{j=1}^{n-1}\frac{1}{c_{n-1,j}}\right]^T$,
$H_2 = \left[\sum_{i=1}^{n-1}\frac{1}{c_{i1}}, \ldots, \sum_{i=1}^{n-1}\frac{1}{c_{i,n-1}}\right]^T$ and
$H_3 = \left[\frac{1}{c_{n1}}, \ldots, \frac{1}{c_{n,n-1}}\right]^T$,
and the $(n-1)$-dimensional vector $\zero = [0,\ldots,0]$.
Hence, $(x_1, x_2,\ldots, x_{n-1}, y_1, y_2, \ldots, y_{n-1}, z)^T=  G^{-1} \be^{(2n-1)}_{11}$.
Meanwhile, we know that
\begin{align*}
 \norm{(\be_{11})_{||\cS_{\bc}}}_{\bc}^2&= \langle (\be_{11})_{||\cS_{\bc}}, (\be_{11})_{||\cS_{\bc}}\rangle_{\bc}\\
 &=\langle \be_{11}-(\be_{11})_{\bot\cS_{\bc}}, \be_{11}-(\be_{11})_{\bot\cS_{\bc}}\rangle_{\bc}\\
 &=\langle \be_{11}, \be_{11}\rangle_{\bc}-2 \langle \be_{11}, (\be_{11})_{\bot\cS_{\bc}}\rangle_{\bc}+ \langle (\be_{11})_{\bot\cS_{\bc}}, (\be_{11})_{\bot\cS_{\bc}}\rangle_{\bc}
 \\&= \langle \be_{11}, \be_{11}\rangle_{\bc}-\langle \be_{11}, (\be_{11})_{\bot\cS_{\bc}}\rangle_{\bc} ,
\end{align*}
where the last equality is due to the fact that $\langle \be_{11}, (\be_{11})_{\bot\cS_{\bc}}\rangle_{\bc}= \langle (\be_{11})_{\bot\cS_{\bc}}, (\be_{11})_{\bot\cS_{\bc}}\rangle_{\bc}$.
This together with the above yields
$$\norm{(\be_{11})_{||\cS_{\bc}}}_{\bc}^2=z+x_1+y_1.$$

It is easy to see that the same approach can be taken for all the other $\be_{ij}$.
Specifically, we need to identify the vector $(\be_{ij})_{\bot\cS_{\bc}}$ that satisfies $\langle (\be_{ij})_{\bot\cS_{\bc}}, B_{ij}\rangle_\bc=c_{ij}$ and $\langle (\be_{ij})_{\bot\cS_{\bc}}, B_{k\ell}\rangle_\bc=0$ for all other $(k, \ell)$.
Here, $(\be_{ij})_{\bot\cS_{\bc}}$ can have a similar representation as in \eqref{eqn:bot_projection} with the position of $1$ changed from $(1,1)$ to $(i,j)$.
In addition, the variables $x,y,z$ will satisfy a system of equations $G (x^{ij}_1, x^{ij}_2,\ldots, x^{ij}_{n-1}, y^{ij}_1, y^{ij}_2, \ldots, y^{ij}_{n-1}, z^{ij})^T= \be^{(2n-1)}_{ij}$
where $G$ is the same matrix defined above and where $\be^{(2n-1)}_{ij}$ is a $(2n-1)$-dimensional vector with the $i$th and $(n+j)$th elements being $1$, and all others being $0$ when $j\le n-1$,
and with only the $i$th element being $1$ and all others being $0$ when $j=n$.
Hence, following along the lines of the above approach, we have in general
\begin{equation*}
\norm{(\be_{ij})_{||\cS_{\bc}}}_{\bc}^2 = z^{ij}+x^{ij}_1+y^{ij}_1
 \end{equation*}
with $( x_1^{ij}, \ldots, x_{n-1}^{ij}, y_1^{ij}, \ldots, y_{n-1}^{ij}, z^{ij})^T$ given by $ G^{-1} \be^{(2n-1)}_{ij}$.

\begin{remark}
For the $n=2$ case, the explicit expression above can be used to recover the heavy-traffic limit of
$(1/2)\sum\lij\sigma\lij^2c\lij (1-[c\lij^2/\sum_{i'j'}c_{i'j'}^2] )$ in~\cite{LuMaSq+18} for the RHS of \eqref{eq:thm:HTL}.
\end{remark}

\subsection{Alternative Approach for Explicit Heavy Traffic Limit Expression}

We next present an alternative approach to obtain an explicit expression for the projection under the $\bc$-inner product.
The key will be the following lemma.

\begin{lemma}
Suppose that we have $k\le m$ linearly independent vectors, under the $\bc$-inner product, in ${\mathbb R}^m$: $\bv_1, \bv_2, \ldots, \bv_k$.
For any vector $\bx\in {\mathbb R}^m$, its projection to the subspace spanned by $\bv_1, \bv_2, \ldots, \bv_k$ can be expressed as
\begin{align}
\label{eqn:projection_c}
Z  (Z^T \mbox{diag}(\bc) Z)^{-1} Z^T \mbox{diag}(\bc) \bx,
\end{align}
where $Z$ is an $m\times k$ matrix whose columns are $\bv_i, i=1,\ldots, k$, i.e., $Z$ is the matrix that is ``stacked'' by the basis vectors of the subspace,
and $\mbox{diag}(\bc)$ is the $m\times m$ diagonal matrix with the elements of $\bc$. 
\end{lemma}
\begin{proof}
We write the projection as $\by= Z \bu $, where $\bu$ is a $k$-dimensional vector.
The projection means that $\langle \bx -  Z\bu, \bv_i \rangle_\bc=0$ for any $i=1,\ldots, k$,
which in matrix form implies
\begin{align*}
Z^T \mbox{diag}(\bc) (\bx- Z\bu)=0.
\end{align*}
Hence,
\begin{align*}
\bu = (Z^T \mbox{diag}(\bc) Z)^{-1} Z^T \mbox{diag}(\bc) \bx,
\end{align*}
and it is well-known that the $k\times k$ matrix $Z^T \mbox{diag}(\bc) Z$ is invertible.
The expression \eqref{eqn:projection_c} then follows immediately.
\end{proof}

Therefore, for our problem, where $m=n^2$ and $k=2n+1$,
suppose that $Z$ is the $n^2\times(2n-1)$ matrix ``stacked'' by $\be_\bc^{(1)}, \ldots, \be_\bc^{(n)}, {\tilde \be}_\bc^{(1)}, \ldots, {\tilde \be}_\bc^{(n-1)}$.
We then have, for any $\be_{ij}$, the following expression for $(\be_{ij})_{||\cS_{\bc}}$:
\begin{align*}
(\be_{ij})_{||\cS_{\bc}}=Z(Z^T \mbox{diag}(\bc) Z)^{-1} Z^T \mbox{diag}(\bc) \be_{ij},
\end{align*}
which renders
\begin{align*}
\norm{(\be_{ij})_{||\cS_{\bc}}}_{\bc}^2& = 
\langle (\be_{ij})_{||\cS_{\bc}}, (\be_{ij})_{||\cS_{\bc}}\rangle_{\bc}= \langle (\be_{ij}), (\be_{ij})_{||\cS_{\bc}}\rangle_{\bc}=\be_{ij}^T\mbox{diag}(\bc) Z(Z^T \mbox{diag}(\bc) Z)^{-1} Z^T \mbox{diag}(\bc) \be_{ij}.
\end{align*}

\section{Conclusions}
%
%
In this paper we considered the optimal control of $n\times n$ IQSs under the $\bc$-weighted MaxWeight algorithm,
with the goal of gaining fundamental insights on the delay-optimal properties of this cost-weighted variant of MaxWeight scheduling in real-world IQSs.
We established theoretical properties that include showing the $\bc$-weighted MaxWeight algorithm exhibits optimal scaling in heavy traffic
under an objective function consisting of a general linear function of the steady-state average QLs.
Our results shed light on the delay optimality of variants of MaxWeight scheduling and
are expected to be of theoretical interest more broadly than IQSs.


\appendix
\section{Additional Proofs for Heavy Traffic Analysis}

\subsection{Proof of State Space Collapse}
\label{app:SSC}
To simplify the notation, we use $\exq[\,\cdot\,]$ to denote $\ex[\,\cdot\,|\bQ(t)=\bQ]$ throughout this section.
\begin{lemma}
For
Lyapunov function drift $\Delta W_{\bot_\bc}(\bQ)$ $\eqdef$ $\big(W_{\bot_{\bc}}(\bQ(t+1))-W_{\bot_{\bc}}(\bQ(t))\big)\mathbb I_{\{\bQ(t)=\bQ\}}$, we have
\begin{align}
	\pr[|\Delta W_{\bot_{\bc}}(\bQ)|\leq D] = 1, &\qquad \forall \bQ , \label{eq:condition_c2}\\
	\exq[\Delta W_{\bot_{\bc}}(\bQ)] \le -\eta, &\qquad \forall \bQ ,
W_{{\bot}_{\bc}}(\bQ)\geq \kappa, \label{eq:condition_c1}
\end{align}
for some positive numbers $\eta$, $\kappa$ and $D$ that depend on $\widetilde{\sigma}$, $\vnu$, $A_{\max}$ and $\numin$, but not on $\epsilon$.
\label{lem:Lyapunov-drift}
\end{lemma}
\begin{proof}
First of all, \eqref{eq:condition_c2} follows from
\begin{align*}
	|\Delta W_{\bot_{\bc}}(\bQ)|&\leq \bigg|\|\bQ_{\bot_{\bc}}(t+1)\|_{\bc}-\|\bQ_{\bot_{\bc}}(t)\|_{\bc}\bigg| \\
	&\leq \|\bQ(t+1)-\bQ(t)\|_{\bc}\\
	&=\sqrt{\sum_{ij}c_{ij}(Q_{ij}(t+1)-Q_{ij}(t))^2} \\
	&\leq \sqrt{\sum_{ij}c_{ij}A_{ij}^2}
	\leq n\sqrt{c_{\max}}\,A_{\max},
\end{align*}
with $D=n\sqrt{c_{\max}}\,A_{\max}$.
To prove~\eqref{eq:condition_c1} we start with a version of Lemma~$4$ in \cite{maguluri2016}, which can be shown to hold more generally for the new dot product
by appropriately adapting the arguments in the proof of Lemma~$7$ in \cite{ErySri12}.
\end{proof}

\begin{lemma} For all $\bQ \in {\mathbb R}^{n^2}$, we have
	\begin{align}
		\label{eqn:drift_dec} 	\Delta W_{\bot_\bc}(\bQ) \le \frac{1}{2||\bQ_{\bot_\bc}||_\bc} (\Delta V(\bQ) - \Delta V_{\parallel_\bc}(\bQ)),
	\end{align}
	where $V(\bQ)\eqdef \norm{\bQ}_{\bc}^2, V_{\parallel_{\bc}}(\bQ)\eqdef \norm{\bQ_{\parallel_{\bc}}}_{\bc}^2$ and
	\begin{align*}
		\Delta V(\bQ) &\eqdef \big(V(\bQ(t+1))-V(\bQ(t))\big)\mathbb I_{\{\bQ(t)=\bQ\}}\\
		\Delta V_{\parallel_{\bc}}(\bQ) & \eqdef
		\big(V_{\parallel_{\bc}}(\bQ(t+1))-V_{\parallel_{\bc}}(\bQ(t))\big)\mathbb I_{\{\bQ(t)=\bQ\}}.
	\end{align*}
\end{lemma}
\begin{proof}
Let us separately consider the two quantities $\Delta V(\bQ)$ and $\Delta V_{\parallel_\bc}(\bQ)$,
recalling the queueing dynamics in \eqref{eq:dynamics_of_Q-wU}.
For the first quantity, we obtain
\begin{align*}
\exq[\Delta V(\bQ)] & = \exq[||\bQ(t)+ \bA(t) -\bS(t)||_\bc^2 - ||\bU(t)||_\bc^2 -||\bQ(t)||_\bc^2] \\
& \le \exq[||\bA(t) -\bS(t)||_\bc^2  + 2\langle \bQ(t) , \bA(t) -\bS(t)\rangle_\bc] \\
&= \exq\bigg[\sum_{ij} c_{ij}A^2_{ij}(t)+ c_{ij}S_{ij}(t) - 2c_{ij}A_{ij}(t)S_{ij}(t) \bigg] + 2 \langle \bQ, \vlam -\exq[\bS(t)] \rangle_\bc \\
& \le \sum_{ij} c_{ij} (\g_{ij} +\sigma_{ij}^2) + \sum_{ij}c_{ij}S_{ij}(t) -2\epsilon \langle \bQ, \bnu \rangle_\bc + 2 \min \langle \bQ, \bnu - \br\rangle_\bc ,
\end{align*}
where we exploit the facts that $\langle \bQ(t+1), \bU(t)\rangle_\bc = 0$ and that arrivals are independent of the QL and service processes in each time slot,
together with our definition of the $\bc$-weighted MaxWeight algorithm.
The selection of $\br$ will be $\bnu + \frac{\nu^c_{\min}}{||\bQ_{\bot_\bc}||_\bc}\bQ_{\bot_\bc}$,
where $\bnu$ is an arrival rate vector that resides on the boundary of the capacity region with all input and output ports saturated
and where $\nu^c_{\min} \eqdef\min \frac{\nu_{ij}}{c_{ij}}$.
This selection of $\br$ guarantees that it is within the capacity region,
which is readily verified by first observing $\nu_{ij} + \frac{\nu^c_{\min}}{||\bQ_{\bot_\bc}||_\bc}\bQ_{\bot_\bc, ij} \ge \nu_{ij} - \nu_{min} \ge 0$
and then observing $\langle \bnu + \frac{\nu^c_{\min}}{||\bQ_{\bot_\bc}||_\bc}\bQ_{\bot_\bc}, \be^{i}\rangle_\bc \le 1$ and
$\langle \bnu + \frac{\nu^c_{\min}}{||\bQ_{\bot_\bc}||_\bc}\bQ_{\bot_\bc}, {\tilde \be}^{j}\rangle_\bc \le 1$.
We therefore have
\begin{align*}
	& \exq[\Delta V(\bQ)]  \le \sum_{ij} c_{ij} (\g_{ij} +\sigma_{ij}^2) + nc_{\max} 
	 -2\epsilon \langle \bQ, \bnu \rangle_\bc - 2\nu^c_{\min}||\bQ_{\bot_\bc}||_\bc ,
\end{align*}
taking advantage of the fact that $\langle \bQ_{\parallel_\bc} , \bQ_{\bot_\bc} \rangle_\bc = 0$.
Turning to the second quantity, we obtain
\begin{align*}
	\exq[\Delta V_{\parallel_\bc}]
=&\exq[ ||\bQ_{\parallel_\bc}(t+1)- \bQ_{\parallel_\bc}(t)||_\bc^2]  
	 + 2\exq[\langle \bQ_{\parallel_\bc}(t), \bQ_{\parallel_\bc}(t+1)-\bQ_{\parallel_\bc}(t)\rangle_\bc] \\
\ge & 2\exq[\langle \bQ_{\parallel_\bc}(t), \bQ_{\parallel_\bc}(t+1)-\bQ_{\parallel_\bc}(t)\rangle_\bc] \\
\ge & 2\exq[\langle \bQ_{\parallel_\bc}(t), \bA(t) -\bS(t) +\bU(t)\rangle_\bc] \\
\ge & 2\langle \bQ_{\parallel_\bc}(t), \vlam \rangle_\bc - 2\exq[\langle \bQ_{\parallel_\bc}(t), \bS(t) \rangle_\bc] \\
=& -2\epsilon\langle \bQ_{\parallel_\bc}(t), \bnu \rangle_\bc - 2\exq[\langle \bQ_{\parallel_\bc}(t), \bS(t)-\bnu \rangle_\bc] \\
=& -2\epsilon\langle \bQ_{\parallel_\bc}(t), \bnu \rangle_\bc ,
\end{align*}
where we again take advantage of the above facts together with $\langle \bQ_{\parallel_\bc}(t) , \bQ_{\bot_\bc}(t+1) \rangle_\bc \leq 0$,
both $\bQ_{\parallel_\bc}$ and $\bU(t)$ being nonnegative componentwise, and properties related to the cone $\calk_{\bc}$ and its spanned space $\cS_{\bc}$.

Upon substituting the above expressions for both quantities into \eqref{eqn:drift_dec}, we have
\begin{align*}
	\exq[\Delta W_{\bot_\bc}(\bQ)]\le & \frac{1}{2|| \bQ_{\bot_\bc}||} \bigg[\sum_{ij} c_{ij} (\g_{ij} +\sigma_{ij}^2) 
+nc_{\max} -2\epsilon \langle \bQ, \bnu \rangle_\bc- 2\nu^c_{\min}||\bQ_{\bot_\bc}||_\bc  + 2\epsilon\langle \bQ_{\parallel_\bc}(t), \bnu \rangle_\bc\bigg]\\
\leq &\frac{\sum_{ij} c_{ij} (\g_{ij} +\sigma_{ij}^2) +nc_{\max} }{|| \bQ_{\bot_\bc}||} - \nu^c_{\min}
- \frac{\epsilon }{|| \bQ_{\bot_\bc}||} \langle \bQ_{\bot_\bc}(t), \bnu \rangle_\bc .
\end{align*}
Given
$\epsilon < \nu^c_{\min}/(2||\bnu||_\bc)$,
then on the set of
$W_{\bot_\bc}(\bQ) \ge 4  (\sum_{ij} c_{ij} (\g_{ij} +\sigma_{ij}^2) + nc_{\max})/\nu^c_{\min}$,
we obtain
\begin{align*}
	\exq[\Delta W_{\bot_\bc}(\bQ)]
\le & \frac{1}{2||\bQ_{\bot_\bc}||} ( \sum_{ij} c_{ij} (\g_{ij} +\sigma_{ij}^2) + nc_{\max} -2\epsilon \langle \bQ, \bnu \rangle_\bc 
	- 2\nu^c_{\min}||\bQ_{\bot_\bc}||_\bc + 2\epsilon\langle \bQ_{\parallel_\bc}(t), \bnu \rangle_\bc) \\
\le& \frac{ \sum_{ij} c_{ij} (\g_{ij} +\sigma_{ij}^2) + nc_{\max}}{2||\bQ_{\bot_\bc}||} - \nu^c_{\min} -  \epsilon ||\bnu||_\bc \\
\le & \frac{ \sum_{ij} c_{ij} (\g_{ij} +\sigma_{ij}^2) + nc_{\max}}{2||\bQ_{\bot_\bc}||} - \frac{\nu^c_{\min}}{2} \le - \frac{\nu^c_{\min}}{4}.
\end{align*}
Hence, \eqref{eq:condition_c1} holds with $\eta = - \nu^c_{\min}/4$.
\end{proof}

\subsection{Proof of Theorem \ref{thm:Heavytraffic}}
\label{app:3.1}
\begin{lemma}
In the limit as $\epsilon\downarrow 0$, we have
\begin{equation*}
n\ex\bigg[\bigg \langle \bQ_{||\cS_{\bc}}, (\bS(\bQ)-\bA)_{||\cS_{\bc}}\bigg \rangle_{\bc}\bigg] =
	\lim_{\epsilon\to 0} \epsilon \ex\bigg[  \sum_{ij} c\lij \overline{Q}\lij^{(\epsilon)} \bigg] .
\end{equation*}
\label{lem:LHS}
\end{lemma}
\begin{proof}
The LHS of
above
can be written as
\ba
2\ex\bigg[\bigg \langle \bQ_{||\cS_{\bc}}, (\bS(\bQ)-\bA)_{||\cS_{\bc}}\bigg \rangle_{\bc}\bigg]
&= 2\epsilon \bE\bigg[\bigg \langle \bQ_{||\cS_{\bc}}, \vnu \bigg \rangle_{\bc} \bigg]+2\bE\bigg[\bigg \langle \bQ_{||\cS_{\bc}}, \bS (\bQ)-\vnu \bigg \rangle_{\bc}\bigg]\\
&= \frac{2}{n}\epsilon \bE\bigg[\bigg \langle_{\bc} \bQ_{||\cS_{\bc}}, \vone \bigg \rangle_{\bc} \bigg]+2\epsilon \bE\bigg[\bigg \langle \bQ_{||\cS_{\bc}}, \vnu-\frac{1}{n}\vone \bigg \rangle_{\bc} \bigg]\\
&=\frac{2}{n}\epsilon \bE\bigg[\bigg \langle \bQ, \vone \bigg \rangle_{\bc} \bigg]-\frac{2}{n}\epsilon \bE\bigg[\bigg \langle \bQ_{\bot\cS_{\bc}}, \vone \bigg \rangle_{\bc} \bigg],
\ea
where the second equality follows from the fact that $\bS(\overline{\bQ}), \vnu \in \cF$,
and therefore $\bS(\overline{\bQ})-\vnu$ is orthogonal to the space spanned by the normal vectors of $\cF$, i.e., to the space $\cS_{\bc}$;
and the next to last equality follows from the fact that $\vnu, \vone/n \in \cF$.
Since the second term of the last equation goes to $0$ as $\epsilon\downarrow 0$ by the state space collapse from Proposition \ref{prop:SSC}, we have
\begin{equation*} 
	\lim_{\epsilon\downarrow 0} \ex\bigg[\Big \langle \bQ_{||\cS_{\bc}}, (\bS(\bQ)-\bvA)_{||\cS_{\bc}}\Big \rangle_{\bc}\bigg]
	\!=\!\lim_{\epsilon\downarrow 0}\!\frac{\epsilon}{n}\ex\bigg[  \sum_{ij} c\lij \overline{Q}\lij \bigg],
\end{equation*}
thus yielding the LHS of
\eqref{eq:thm:HTL} in Theorem~\ref{thm:Heavytraffic}.
\end{proof}

\begin{lemma}
In the limit as $\epsilon\downarrow 0$, we have
\begin{align}		
 \ex\bigg[\norm{(\bA-\bS(\bQ))_{||\cS_{\bc}} }_{\bc}^2\bigg] -\ex\bigg[\norm{\bU_{||\cS_{\bc}}(\bQ)}_{\bc}^2\bigg] 
&  +2\ex\bigg[\bigg \langle\bQ_{||\cS_{\bc}}^+, \bU_{||\cS_{\bc}}(\bQ)\bigg \rangle_{\bc}\bigg] = \frac{n}{2}\bigg \langle \vsig^2, \vzeta  \bigg \rangle_{\bc}.
\label{eq:RHS}
\end{align}
\label{lem:RHS}
\end{lemma}
\begin{proof}
First of all, from the equation \begin{align*}
0&=\ex\bigg[\sum_{i,j}Q_{ij}(t+1) - \sum_{ij}Q(t) | \bQ(t)= \overline{\bQ}\bigg]=\ex\bigg[\sum_{i,j}A_{ij}-\sum_{i,j}S_{i,j}(\overline{\bQ})-\sum_{i,j} U_{i,j}\bigg],
\end{align*} we can conclude that $\ex[\sum_{i,j} U_{ij}(\bQ)] = n\epsilon$,
which implies that the second term on the LHS of \eqref{eq:RHS} converges to $0$ as $\epsilon\downarrow 0$:
\begin{align*}
	\ex\bigg[\norm{\bU_{||\cS_{\bc}}(\bQ)}_{\bc}^2 \bigg] &
	\!\!\leq\! \ex\bigg[\sum_{i,j} c_{ij} U_{ij}(\overline{\bQ})^2 \bigg]
	\!\!=\!\ex\bigg[\sum_{i,j} c_{ij} U_{ij}(\overline{\bQ})\bigg]
	\leq c_{\max}n\epsilon\to 0,\quad\textrm{as $\epsilon\downarrow 0$}.
\end{align*}

For the third term on the LHS of \eqref{eq:RHS}, we have
\begin{align*}
	2\ex\bigg[\Big \langle\bQ_{||\cS_{\bc}}^+, \bU_{||\cS_{\bc}}(\bQ)\Big \rangle_{\bc}\bigg] & = 2\ex\bigg[\Big \langle\bQ^+, \bU_{||\cS_{\bc}}(\bQ)\Big \rangle_{\bc}\bigg]
	-2\ex\bigg[\Big \langle\bQ_{\bot\cS_{\bc}}^+, \bU_{||\cS_{\bc}}(\bQ)\Big \rangle_{\bc}\bigg] 
	&= -2\ex\bigg[\Big \langle\bQ_{\bot\cS_{\bc}}^+, \bU_{||\cS_{\bc}}(\bQ)\Big \rangle_{\bc}\bigg],
\end{align*}
because
$\overline{Q}^+_{ij}=0$ if $U_{ij}(\overline{Q})=1$.  Furthermore,
\begin{align*}
	\bigg|\ex\bigg[\bigg \langle\bQ_{||\cS_{\bc}}^+, \bU_{||\cS_{\bc}}(\bQ)\bigg \rangle_{\bc}\bigg]\bigg| 
	& \leq \sqrt{\ex\bigg[ \norm{\bQ_{\bot\cS_{\bc}}^+}^2\bigg]\ex\bigg[ \norm{\bU_{||\cS_{\bc}}(\bQ)}^2 \bigg]}\leq M_2 \sqrt{\ex\bigg[ \norm{\bU_{||\cS_{\bc}}(\bQ)}^2 \bigg]} \leq M_2 \sqrt{2n\epsilon},
\end{align*}
where the first inequality is just Cauchy-Schwartz,
$M_2$ is the constant in Proposition~\ref{prop:SSC}, and the last inequality is due to
$\ex[\sum_{i,j} U_{ij}(\bQ)] = n\epsilon$.
This then implies that the third term also converges to $0$ as $\epsilon\downarrow 0$.

Finally, turning to investigate the first term, let $f_1, f_2, \ldots, f_{2n-1}$ be an orthonormal base for space $\cS$.
Then, from basic properties of the space, there exist $v_{\ell, i}$ and ${\tilde v}_{\ell, j}$ such that $f_{\ell i j} = \frac{v_{\ell i}+ {\tilde v}_{\ell j}}{c_{ij}}$.
Thus, we can derive
\begin{align*}
\ex[||(\bA-\bS(\bQ))_{||\cS_{\bc}} ||^2] &=   \sum_{\ell=1}^{2n-1} \E\left[\langle \bA-\bS(\bQ), f_\ell \rangle_{\bc}^2  \right]
=  \sum_{\ell=1}^{2n-1} \E \left[\bigg( \sum_{ij}(A_{ij}-S_{ij}) \bigg(\frac{v_{\ell i}+ {\tilde v}_{\ell j}}{c_{ij}}\bigg)c_{ij}\bigg)^2\right]\\
	&= \sum_{\ell =1}^{2n-1} \var \bigg[\sum_i v_{\ell i} \sum_j A_{ij}+ \sum_j {\tilde v}_{\ell j}\sum_i A_{ij} \bigg]\\
	&=  \sum_{\ell =1}^{2n-1} \bigg[ \sum_i v_{\ell i}^2 \sum_j \sigma_{ij}^2 + \sum_j {\tilde v}^2_{ij} \sum_j \sigma_{ij}^2 + 2\sum_{ij} v_{\ell i}{\tilde v}_{\ell j}\sigma_{ij}^2\bigg] \\
	&= \sum_{ij} c_{ij} \sigma_{ij} \sum_{\ell=1}^{2n-1} \bigg(\frac{v_{\ell i}+ {\tilde v}_{\ell j}}{c_{ij}}\bigg)^2 c_{ij}
	=\sum_{ij} c_{ij} \sigma_{ij} \sum_{\ell=1}^{2n-1} \langle f_{\ell}, e_{ij}\rangle_{\bc}^2\\
	&= \sum_{ij} c_{ij} \sigma_{ij} || (e_{ij})_{\parallel_{\cS_{\bc}}} ||^2 =\bigg \langle \vsig^2, \vzeta  \bigg \rangle_{\bc},
\end{align*}
which establishes the desired result.
\end{proof}

\bibliographystyle{abbrv}

\begin{thebibliography}{1}

\bibitem{ErySri12}
A.~Eryilmaz, R.~Srikant.
\newblock Asymptotically tight steady-state queue length bounds implied by
  drift conditions.
\newblock {\em QUESTA},
2012.


\bibitem{LuMaSq+18}
Y.~Lu, S.~T. Maguluri, M.~S. Squillante, T.~Suk.
\newblock Optimal Dynamic Control for Input-Queued Switches in Heavy Traffic.
\newblock {\em Proc.\ ACC}, 2018.


\bibitem{maguluri2016b}
S.~T. Maguluri, S.~K. Burle, R.~Srikant.
\newblock Optimal heavy-traffic queue length scaling in an incompletely
  saturated switch.
\newblock {\em Preprint}, 2016.

\bibitem{maguluri2016}
S.~T. Maguluri, R.~Srikant.
\newblock Heavy traffic queue length behavior in a switch under the maxweight
  algorithm.
\newblock {\em Stoch.\ Syst.},
2016.

\bibitem{MCKEOWN96}
N.~McKeown, V.~Anantharam, J.~Walrand.
\newblock Achieving 100\% throughput in an input-queued switch.
\newblock {\em Proc.\ INFOCOM},
1996.

\bibitem{perry2014fastpass}
J.~Perry, A.~Ousterhout, H.~Balakrishnan, D.~Shah, H.~Fugal.
\newblock Fastpass: A centralized zero-queue datacenter network.
\newblock {\em ACM SIGCOMM CCR}, 44(4): 307--318, 2014.

\bibitem{singh2015jupiter}
A.~Singh et al.
\newblock Jupiter rising: A decade of clos topologies and centralized control in Google's datacenter network,
\newblock {\em ACM SIGCOMM CCR}, 45(4): 183--197, 2015.

\bibitem{stolyar2004}
A.~L. Stolyar.
\newblock Maxweight scheduling in a generalized switch: State space collapse
  and workload minimization in heavy traffic.
\newblock {\em Ann.\ Appl.\ Probab.},
2004.

\bibitem{TasEph92}
L.~Tassiulas, A.~Ephremides.
\newblock Stability properties of constrained queueing systems and scheduling
  policies for maximum throughput in multihop radio networks.
\newblock {\em IEEE TAC},
1992.

\end{thebibliography}

%


\end{document}